\documentclass[12pt]{amsart} 
\usepackage{amssymb,color,verbatim}
\newtheorem{theorem}{Theorem}
\newtheorem{cor}[theorem]{Corollary}

\newtheorem{lemma}[theorem]{Lemma}
\newtheorem{pro}[theorem]{Proposition}

\newtheorem{remark}[theorem]{Remark}

\newcommand{\bs}{\bigskip}
\newcommand{\n}{\noindent}
\newcommand{\ds}{\displaystyle}

\newcommand{\Aa}{\ensuremath{\mathcal{A}}}
\newcommand{\Ba}{\ensuremath{\mathcal{B}}}
\newcommand{\Pa}{\ensuremath{\mathcal{P}}}

\newcommand{\Da}{\ensuremath{\mathcal{D}}}

\newcommand{\Ha}{\ensuremath{\mathcal{H}}}
\newcommand{\ta}{\Theta}
\begin{document}
\title[The structure of the minimum size supertail of a subspace partition]{The structure of the minimum size supertail of a subspace partition}
\author[E. N\u{a}stase and P. Sissokho]{\tiny{E. N\u{a}stase} \\ \\
 Mathematics Department\\ Xavier University\\
Cincinnati, Ohio 45207, USA\\ \\
P. Sissokho  \\ \\
Mathematics Department \\ Illinois State University\\ 
Normal, Illinois 61790, USA}
\thanks{nastasee@xavier.edu,  psissok@ilstu.edu}
\dedicatory{\scriptsize{(Dedicated to Professor Olof Heden on the occasion of his retirement)}}

\begin{abstract}
Let $V=V(n,q)$ denote the vector space of dimension $n$ over the finite field with $q$ elements.
A {\em subspace partition} $\Pa$ of $V$ is a collection of nontrivial subspaces of $V$ such that 
each nonzero vector of $V$ is in exactly one subspace of $\Pa$. 
For any integer $d$, the {\em $d$-supertail} of $\Pa$ is the set of subspaces in $\Pa$ of dimension less than $d$, and it is denoted by $ST$. Let $\sigma_q(n,t)$ denote the minimum number of subspaces 
in any subspace partition of $V$ in which the largest subspace has dimension $t$. It was shown by Heden 
et al. that $|ST|\geq \sigma_q(d,t)$, where $t$ is the largest dimension of a subspace in $ST$.
In this paper, we show that if $|ST|=\sigma_q(d,t)$, then the union of all the subspaces in $ST$ constitutes a subspace under certain conditions. 
\end{abstract}

\maketitle
\section{Introduction}\label{sec:1}
Let $V=V(n,q)$ denote a vector space of dimension $n$ over the finite field with $q$ elements. We use 
the term {\em $d$-subspace} to refer to a subspace of dimension $d$.
For any subspace $U$ of $V$, we let $U^*$ denote the set of nonzero vectors in $U$. 
A {\em subspace partition} $\Pa$ of $V$, also known as a {\em vector space partition}, is a collection of 
nontrivial subspaces of $V$ such that each vector of $V^*$ is in exactly one subspace of 
 $\Pa$ (e.g., see Heden~\cite{He-S} for 
a survey). The study of subspace  partitions originated from the general problem of partitioning a 
finite (not necessarily abelian) group into subgroups that only intersect at the identity element (e.g.; 
see Zappa~\cite{Za} for a survey). Subspace partitions can be used to construct translation 
planes, error-correcting codes, orthogonal arrays, and designs (e.g., 
see~\cite{An,Ba,Be,DrFr,He-C,HeSc,HoPa,Li}).
 
\bs Suppose that there are $m$ distinct dimensions, $d_1<d_2<\dots<d_m$, that occur in a subspace 
partition $\Pa$, and let $n_d$ denote the number of $d$-subspaces in $\Pa$. 
Then the expression 
$[d_1^{n_{d_1}},\ldots,d_m^{n_{d_m}}]$ is called the {\em type} of $\Pa$. The general 
problem in this area is to find necessary and sufficient conditions for the existence of 
a subspace partition of $V$ of a given type (e.g.,  see~\cite{Bu,Be2,ESSSV,He-E,Li,HeLe} 
for the solution of some special cases). Two obvious necessary conditions for the existence 
of a subspace partition of type $[d_1^{n_{d_1}},\ldots,d_m^{n_{d_m}}]$ are 
the {\em packing condition} 
\begin{equation}\label{eq:pack}
\sum_{i=1}^{m} n_{d_i}(q^{d_i}-1)=q^n-1,
\end{equation}
and the {\em dimension condition}
\begin{equation}\label{eq:dim}
\begin{cases} 
&\mbox{$n\geq d_i+d_j$ if $n_{d_i}$+$n_{d_j}\geq2$ and $i\not=j$}; \mbox{ and } \\
&\mbox{$n\geq 2d_i$ if $n_{d_i}\geq2$}.
\end{cases}
\end{equation}

\bs To the best of our knowledge, there are not many other known  necessary 
conditions for the existence of a subspace partition $\Pa$ of $V$.
Heden and Lehmann~\cite{HeLe} derived some necessary conditions 
(see Lemma~\ref{lem:HeLe1}) by essentially counting in two ways 
tuples of the forms $(H,U)$ and $(H,W_1,W_2)$, where $H$ is a hyperplane of $V$ and $U,W_1,W_2$ 
are subspaces of $\Pa$ that are contained in $H$.
Blinco et al.~\cite{BESSSV} and Heden~\cite{He-E,He-T} derived some necessary conditions 
on the set $T$ of subspaces of minimum dimension (called {\em tail} in~\cite{He-E}) of $\Pa$.
The concept of tail was later generalized by Heden et al.~\cite{HLNS2}, as we shall see below.

Let $\Pa$ be a subspace partition of $V=V(n,q)$ of type $[d_1^{n_{d_1}},\ldots,d_m^{n_{d_m}}]$.  
For any integer $d$ such that $d_1<d\leq d_m$, the {\em $d$-supertail} of $\Pa$ is the set of 
subspaces in $\Pa$ of dimension less than $d$, and it is denoted by $ST$.
The {\em size} of a subspace partition $\Pa$ is the number of subspaces in $\Pa$. 
For $1\leq t<n$, let $\sigma_q(n,t)$ denote the {\em minimum size} of any subspace partition of 
$V$ in which the largest subspace has dimension $t$. 
The exact value of $\sigma_q(n,t)$ is given by the following theorem (see Andr\'e~\cite{An} and Beutelspacher~\cite{Be} for $n\pmod{t}\equiv 0$, and see~\cite{HLNS1,NaSi} for $n\pmod{t}\not\equiv 0$).
\begin{theorem}\label{thm:HLNS1}  
Let $n, k, t$, and $r$ be integers such that 
$0\leq r < t$, $k\geq 1$, and $n=kt+r$. Then
\[\sigma_q(n,t)=\begin{cases} 
\ds\frac{q^{kt}-1}{q^t-1}\;\mbox{ for $r= 0$},\\
q^{t}+1\;\mbox{ for $r\geq 1$ and $3\leq n<2t$},\\
  q^{t+r} \sum\limits_{i=0}^{k-2}q^{it}+q^{\lceil \frac{t+r}{2}\rceil}+1\mbox{ for $r\geq 1$ and $n\geq2t$}.
\end{cases}\]
\end{theorem}

The following theorem of Heden et al.~\cite{HLNS2} generalizes a theorem of Heden~\cite[Theorem~$1$]{He-T}.
\begin{theorem}\label{thm:HLNS2}
Let $\Pa$ be a subspace partition of $V(n,q)$ of type $[d_1^{n_{d_1}},\ldots,d_m^{n_{d_m}}]$ and let $2\le s\le m$. If $ST$ is $d_s$-supertail of $\Pa$, then 
\begin{equation}\label{eq:main}
|ST|\geq \sigma_q(d_s,d_{s-1})\;.
\end{equation}
\end{theorem}
If equality holds in~\eqref{eq:main}, then Theorem~\ref{thm:HLNS2} has the following interesting corollary (see~\cite{HLNS2}). 
\begin{cor}\label{cor:HLNS2}
If $|ST|=\sigma_q(d_s,d_{s-1})$ and $d_s\geq 2d_{s-1}$, then the union of the subspaces in $ST$ forms a $d_s$-subspace. 
\end{cor}
Note that the crucial part of the conclusion of Corollary~\ref{cor:HLNS2} is that  
the set of all points covered by the subspaces in $ST$ is a subspace. (In general, the
 $d_s$-supertail of a subspace  partition of $V(n,q)$ need not be a subspace.)
One outstanding question that remains is whether the conclusion of Corollary~\ref{cor:HLNS2} 
holds for $d_{s-1}<d_s<2d_{s-1}$.
For the special case when $ST$ is a simple tail, Heden~\cite[Theorem~$3$]{He-T} proved the following theorem.
\begin{theorem}\label{thm:He-T}
Let $\Pa$ be a subspace partition of $V(n,q)$ of type $[d_1^{n_{d_1}},\ldots,d_m^{n_{d_m}}]$.
If $ST$ is the tail of $\Pa$ $($i.e., $ST$ is the set of $d_1$-subspaces$)$  
 such that $|ST|= q^{d_1}+1$ and $d_2<2{d_1}$, then $ST$ is a $d_1$-spread $($ i.e., a subspace partition consisting of  $d_1$-subspaces$)$.
\end{theorem}
In this paper we prove the following generalization of Theorem~\ref{thm:He-T}. 
\begin{theorem}\label{thm:main}
Let $\Pa$ be a subspace partition of $V(n,q)$ of type $[d_1^{n_{d_1}},\ldots,d_m^{n_{d_m}}]$. 
Let $2\le s\le m$, and suppose $ST$ is
 a $d_s$-supertail of $\Pa$ such that $|ST|=\sigma_q(d_s,d_{s-1})$ and $d_s< 2d_{s-1}$.
Furthermore, assume that one of the following conditions holds.
\begin{enumerate}
\item[(i)] $s-1\leq 2$, that is $ST$ contains subspaces of at most $2$ different dimensions.
\item[(ii)] $d_s=2d_{s-1}-1$.
\item[(iii)] All the subspaces in $\Pa\setminus ST$ have the same dimension.
\end{enumerate}
Then the union of the subspaces in $ST$ forms a subspace $W$.  Moreover,
\begin{enumerate}
\item[(a)] $s-1=1$, $n_1=q^{d_1}+1$, and $\dim W =2d_1$, or
\item[(b)] $s-1=2$, $n_1=q^{d_2}$, $n_2=1$, and $\dim W=d_1+d_2$.
\end{enumerate}
\end{theorem}

The following result is a consequence of Theorem~\ref{thm:HLNS2}, Corollary~\ref{cor:HLNS2}, and Theorem~\ref{thm:main}.
\begin{cor}\label{cor:main-cor}
Let $\Pa$ be a subspace partition of $V(n,q)$ of type $[d_1^{n_{d_1}},\ldots,d_m^{n_{d_m}}]$ and let $2\le s< m$.
Let $ST$ be the $d_s$-supertail of $\Pa$, let $\widehat{ST}$ be its $d_{s+1}$-supertail, and 
assume that $|ST|=\sigma_q(d_s,d_{s-1})$. 
\begin{enumerate}
\item[(i)] If $2\leq s\leq3$, $d_s<2d_{s-1}$, and $d_{s+1}< 2d_s$, then
\[|\widehat{ST}|\geq \sigma_q(d_{s+1},d_{s})+\sigma_q(d_s,d_{s-1}).\]
\item[(ii)]If $d_s\geq 2d_{s-1}$, or if $s=3$ and $d_3=d_2+d_1$, then
\[|\widehat{ST}|\geq \sigma_q(d_{s+1},d_{s})+\sigma_q(d_s,d_{s-1})-1.\]
\end{enumerate}
\end{cor}
\begin{remark}\label{rmk:main-cor} 
Note that the condition $s=3$ in Corollary~\ref{cor:main-cor}(ii)
is equivalent to saying that $ST$ contains subspaces of two different dimensions, namely $d_1$ and $d_2$.
\end{remark}

Theorem~\ref{thm:main} can be viewed as a special case of the general question of determining 
nontrivial conditions under which a set of points of a projective space forms a subspace.
We conjecture that Theorem~\ref{thm:main} holds in all cases, and not just if $(i)$, $(ii)$, or $(iii)$ holds.

\bs The rest of the paper is organized as follows. In Section~\ref{sec:2}, we gather 
some known results that we shall use in Section~\ref{sec:3} to first establish some auxiliary results,  
and then to prove our main results, i.e., Theorem~\ref{thm:main} and Corollary~\ref{cor:main-cor}.
Finally, we include some supporting lemmas in Section~\ref{sec:4} (Appendix).
\section{Preliminaries}\label{sec:2}
Let $n$ be a positive integer and let $q$ be a prime power. Set $\ta _0=0$. For any integer $i$
such that $1\leq i\leq n$, let 
\[\ta _i=\frac{q^i-1}{q-1}\] 
denote the number of points (i.e.,$1$-subspaces) in an $i$-subspace.

We will need the following elementary results (Proposition~\ref{pro:num-hyperplanes} 
and Proposition~\ref{pro:is-d-dim}).
\begin{pro}\label{pro:num-hyperplanes}
The number of hyperplanes containing a given $d$-subspace of $V(n,q)$ is $\ta_{n-d}$.
\end{pro}
\begin{pro}\label{pro:is-d-dim}
If $U$ is a subset of $V=V(n,q)$ containing $\ta_d$ points and contained in precisely 
$\ta_{n-d}$ hyperplanes, then $U$ is a $d$-subspace of $V$. 
\end{pro}
To state the next lemmas, we need the following definitions.
For $n\geq 2$, let $\Pa$ be a subspace partition of $V=V(n,q)$ of type $[d_1^{n_{d_1}},\ldots,d_m^{n_{d_m}}]$.
For any hyperplane $H$ of $V$, let $b_{H,x}$ be the number of $x$-subspaces in $\Pa$ that are contained in $H$ and set $b_H=[b_{H,d_1},\ldots,b_{H,d_m}]$.  Define
\[\Ba=\{b_H:\; \mbox{$H$ is a hyperplane  of $V$}\}.\]
For any  $b \in \Ba$, let $s_{b}$ denote the number of hyperplanes $H$ of $V$ such that $b_H=b$.

\bs We will use Lemma~\ref{lem:HeLe1} and Lemma~\ref{lem:HeLe0} by Heden and Lehmann~\cite{HeLe}.
\begin{lemma}\label{lem:HeLe1}
Let $\Pa$ be a subspace partition of $V(n,q)$, and let $\Ba$ and $s_b$ be as defined earlier.
If $\Pa$ contains two different subspaces, one of dimension $d$ and another of dimension $d'$, 
with $1\leq d,d'\leq n-2$, then

\n $(i)$ $\sum\limits_{b\in \Ba} s_b=\ta_n$,

\n $(ii)$ $\sum\limits_{b\in \Ba}b_d\; s_b=n_d \ta_{n-d}$,

\n $(iii)$ $\sum\limits_{b\in \Ba}{b_d\choose 2}s_b={n_d\choose 2}\ta_{n-2d}$,

\n $(iv)$ $\sum\limits_{b\in \Ba}b_db_{d'}s_b=n_d n_{d'}\ta_{n-d-d'}$.
\end{lemma}
\begin{lemma}\label{lem:HeLe0}
Let $\Pa$ be a subspace partition of $V(n,q)$. If $H$ is a hyperplane of $V$, then
\[|\Pa|=1+\sum\limits_{i=1}^{m} b_{H,d_i}q^{d_i}.\]
\end{lemma}
We will also use the following lemma due to Heden et al.~\cite{HLNS1}.
\begin{lemma}\label{lem:pts}
Let $\Pa$ be a subspace partition of $V(n,q)$ of type $[d_1^{n_1},\ldots,d_m^{n_m}]$ and let $2\le s\le m$.
If $ST$ is a $d_s$-supertail of $\Pa$ and $H$ is a hyperplane of $V$, then 
\[\sum_{i=1}^{s-1}(n_{d_i}-b_{H,d_i})q^{d_i}=c_H\cdot q^{d_s},\]
where $c_H=q^{n-d_s}-\sum\limits_{i=s}^m(n_{d_i}-b_{H,d_i})q^{d_i-d_s}$ is a nonnegative integer.
\end{lemma}

Finally, we will need the following lemma  due to Herzog and Sch\"onheim \cite{HeSc}, and independently Beutelspacher~\cite{Be} and Bu~\cite{Bu}.
\begin{lemma}\label{lem:Be}
Let $n$ and $d$ be integers such that $1\leq d \leq n/2$. Then $V(n,q)$ admits 
a partition with one subspace of dimension $n-d$ and $q^{n-d}$ subspaces of dimension $d$.
\end{lemma}
\section{Auxiliary Results and the proof of the main theorem}\label{sec:3}
In this section, we use $\Ha$ to denote the set of all hyperplanes of $V$.

\begin{lemma}\label{lem:A} 
Let $\Pa$ be a subspace partition of $V=V(n,q)$
of type $[d_1^{n_{d_1}},\ldots,d_m^{n_{d_m}}]$, where $1\leq d_1<\ldots<d_m$. Assume that $2\le s\le m$, and let $ST$ be a $d_s$-supertail of $\Pa$ such that $|ST|=\sigma_q(d_s,d_{s-1})$ and $d_s< 2d_{s-1}$. Then $d_s\leq d_{s-1}+d_1$.
\end{lemma}
\begin{proof}
Suppose that $d_s> d_{s-1}+d_1$.
Let $U,W\in ST$ be such that $\dim U=d_{s-1}$ and $\dim W=d_1$.
Let $B_W$ be a basis of $W$, $B_U$ a basis of $U$, and consider 
a basis $B$ of $V$ obtained by extending $B_U\cup B_W$. Then 
$V'=\mbox{span}(B\setminus B_W)$ is a subspace of $V$ such that $\dim V'=n-d_1$, $U\subseteq V'$,
and $V'\cap W=\emptyset$.
Now let $\mathcal{P}'$ be the subspace partition induced by $\mathcal{P}$ in $V'$,
and let $ST'=\{A\cap V'\not=\{\mathbf 0\}:\; A\in ST\}$, where $\mathbf 0$ denotes the zero vector.
Let $X'$ be a subspace in $\mathcal{P}'\setminus ST'$ that has a minimum possible dimension. 
Since $X'=X\cap V'$ for some $X\in \mathcal{P}\setminus ST$ and $\dim X\geq d_s>d_{s-1}+d_1$, it follows that $\dim X'\geq \dim X-d_1>d_{s-1}$.
Moreover, the subspace $U'=U\cap V'=U$ is in $ST'$ and $W'=W\cap V'=\{\mathbf 0\}$.
Thus, $ST'$ is a supertail of  $\mathcal{P}'$ with highest dimension $d_{s-1}=\dim U'$ and 
of size $|ST'|\leq |ST\setminus\{W\}|\leq |ST|-1= q^{d_{s-1}}$. This contradicts the fact that $|ST'|\geq \sigma_q(\dim X',d_{s-1})=q^{d_{s-1}}+1$. The lemma follows.
\end{proof}
\begin{lemma}\label{lem:C} 
Let $\Pa$ be a partition of $V(n,q)$ with supertail $ST$ consisting of subspaces of dimensions 
at most $t$. For any $H\in \Ha$, let $\beta_H=\sum_{i\leq t}b_{H,i} q^{i}$ and let $\beta_0=\min_{H\in\Ha} \beta_H$. 
Then $|ST|\geq \beta_0+1$.
\end{lemma}
\begin{proof}
Suppose that $|ST|\leq \beta_0$. Then, applying the definition of $\beta_H$, implies that
\[\ta_n|ST| \leq\ta_n\beta_0 
=\sum_{H\in \Ha} \beta_0\leq \sum_{H\in \Ha} \beta_H
= \sum_{H\in \Ha} \sum_{i=1}^t b_{H,i}q^i,\]
and thus with the use of Lemma~\ref{lem:HeLe1}, we obtain
\begin{align*}
 \ta_n|ST| \leq\sum_{i=1}^tq^i \Big{(}\sum_{H\in \Ha} b_{H,i}\Big{)} 
 & \leq \sum_{i=1}^tq^i(n_i\ta_{n-i}) \\
 &\leq  \sum_{i=1}^tn_i(\ta_{n}-\ta_i)\leq  \ta_n|ST|- \sum_{i=1}^tn_i\ta_i,
\end{align*}
which is a contradiction since $\ds\sum_{i=1}^tn_i\ta_i$ is the number of points in $ST$,
and this number is positive.
\end{proof}
\begin{lemma}\label{lem:A2} 
Let $\Pa$ be a partition of $V(n,q)$ with a $d$-supertail $ST$ such that $t$ is the maximum 
dimension of any subspace in $ST$ and $d<2t$.
For any $H\in \Ha$, let $\beta_H=\sum_{i=d_1}^t b_{H,i} q^i$ and $\beta_0=\min_{H\in\Ha} \beta_H$.
 If $ST$ has size $q^t+1$, then $\beta_0=q^t$. Moreover, there exists an integer $c_0$ such that 
\[\sum_{i=1}^{t} n_i\ta_i=\frac{c_0q^d-1}{q-1}.\]
\end{lemma}
\begin{proof}
Let $a$ be the minimum dimension of any subspace in $ST$ 
First, suppose that for some $H\in \Ha$, we have $\beta_H<q^t$.
Then by Lemma~\ref{lem:pts}, there exists an integer $c_H$ such that 
$\sum_{i=a} ^t(n_i-b_{H,i})q^i=c_Hq^{d} $. Hence, we have
\begin{eqnarray*} 
\sum_{i=a} ^t n_iq^{i-a}=c_Hq^{d-a}+\beta_Hq^{-a}.
\end{eqnarray*}
Since $\beta_H<q^t$, 
we obtain $\beta_Hq^{-a}<q^{t-a}$. Then, it follows from~\cite[Proof of Proposition~$6$ (Case~2)]{HLNS2}
that $|ST|>q^t+1$. This is a contradiction and thus $\beta_H\geq q^t$ for all $H\in \Ha$ . 
By Lemma~\ref{lem:C}, $\beta_0\leq |ST|-1=q^t$, and thus $\beta_0=q^t$.

 Now by Lemma~\ref{lem:pts}, there exists an integer $c_0$ such that 
$\sum_{i=a}^t n_iq^i-\beta_0=c_0q^{d}$. Since $\beta_0=q^t$ and $\sum_{i=a}^t n_i=|ST|=q^t+1$, it follows after some 
arithmetic that 
\begin{eqnarray}\label{eq:3}
\sum_{i=a}^t  n_i\ta_i=\frac{c_0q^{d}-1}{q-1}.
\end{eqnarray}

\end{proof}
Let $\Pa$ be a partition of $V(n,q)$ with a $d$-supertail $ST$.
If $ST$ has type $[t^{q^t+1}]$, then we recall that it follows 
from Heden~\cite[Theorem~$3$]{He-T} that the union of the subspaces in $ST$
is a $2t$-subspace. The following lemma is an extension of that result.
\begin{lemma}\label{lem:sub}
Let $\Pa$ be a partition of $V(n,q)$ with a $d$-supertail $ST$ of type $[t^1,a^{q^t}]$, 
i.e., $ST$ contains one subspace of dimension $t$ and $q^t$ subspaces of dimension $a$, where $t>a$.
Then the union of the subspaces in $ST$ forms a $t+a$-subspace.
\end{lemma}
\begin{proof}  
Recall that $\Ha$ denotes the set of all hyperplanes of $V$. Let $H\in \Ha$ be any hyperplane.
It follows from Lemma~\ref{lem:pts} that there exists an integer $c_H\geq0$ such that
\[ (n_t-b_{H,t})q^{t}+(n_a-b_{H,a})q^{a}=c_Hq^d.\]
Thus, 
\begin{align}\label{eq:pts}
b_{H,a}=q^t+(1-b_{H,t})q^{t-a}-c_Hq^{d-a},
\end{align} 
where $0\leq b_{H,a}\leq q^t$ and $b_{H,t}\in\{0,1\}$. 
Let $\Aa$ be the set of $a$-subspaces in $ST$, and let $\alpha_i$ denote the number 
of hyperplanes in $V$  that contain exactly $i$ members of $\Aa$. 
If $\alpha_i\neq0$, then there exists a hyperplane $H\in\Ha$ that contains exactly $b_{H,a}=i$ members from $\Aa$.
Thus, it follows from~\eqref{eq:pts} that
\begin{equation}\label{eq:i}
\alpha_i\not=0\;\Rightarrow\; q^{t-a} \mbox{ divides $i$}.
\end{equation}

Define the integers $x$, $y$, and $z$ as follows:
\[x=\sum_{i=q^{t-a}}^{q^t}i\alpha_i,\quad y=\sum_{i=q^{t-a}}^{q^t}{i\choose2}\alpha_i,\quad 
z=\sum_{i=q^{t-a}}^{q^t}\alpha_i.\]
 Then it follows from Lemma~\ref{lem:HeLe1} that
\begin{equation}\label{eq:a}
x=\sum_{i=q^{t-a}}^{q^t} i\alpha_i=n_a \ta_{n-a}
\end{equation} 
 and 
\begin{equation}\label{eq:b}
y=\sum_{i=q^{t-a}}^{q^t}{i\choose2}\alpha_i={n_a\choose2}\ta_{n-2a}.
\end{equation}
Since $\sum_{i=0}^{q^t}\alpha_i=|\Ha|=\ta_n$, it follows by~\eqref{eq:i} and the definition of $z$ that 
\begin{equation}\label{eq:g}
z=\sum_{i=q^{t-a}}^{q^t}\alpha_i=\ta_n-\alpha_0.
\end{equation}
Using~\eqref{eq:a}--\eqref{eq:g} and the fact that $n_a=q^t$ and
$\ta_i=(q^i-1)/(q-1)$, we obtain
\begin{align}\label{eq:ab-1}
&\sum_{i=q^{t-a}}^{q^t}\alpha_i(i-q^{t-a})(i-q^t)\cr
&=2y+(1-q^{t-a}-q^t)x+q^{2t-a}z\cr
&=q^t(q^t-1)\ta_{n-2a}+(1-q^{t-a}-q^t)q^t\ta_{n-a}+q^{2t-a}(\ta_n-\alpha_0)\cr
&=\ta_{n+t-a} - \ta_{n+t-2a}-q^{2t-a}\alpha_0.
\end{align}
Note that
\begin{equation}\label{eq:ab-2}
(i-q^{t-a})(i-q^t)\begin{cases}
=0\;&\mbox{ if } i=q^{t-a},\\
<0\;&\mbox{ if } q^{t-a}<i<q^t,\\
=0\;&\mbox{ if } i=q^t.
\end{cases}
\end{equation}
Thus, it follows from~\eqref{eq:ab-1} and~\eqref{eq:ab-2} that
\begin{equation}\label{eq:ab-3}
\sum_{i=q^{t-a}}^{q^t}\alpha_i(i-q^{t-a})(i-q^t)=\ta_{n+t-a} - \ta_{n+t-2a}-q^{2t-a}\alpha_0\leq 0,
\end{equation}
and it follows from~\eqref{eq:ab-3} that 
\begin{equation}\label{eq:ab-4}
\alpha_0\geq\ta_{n-t}-\ta_{n-t-a}.
\end{equation}
Furthermore, since $b_{H,t} \in \{0,1\}$, it follows from~\eqref{eq:pts} 
that for any hyperplane $H$, we have
\begin{equation}\label{eq:ii}
b_{H,a}=0 \;\Rightarrow\; b_{H,t}=1.
\end{equation}
Hence, if $W_t$ is a $t$-subspace and $W_a$ is an $a$-subspace in the supertail $ST$, then each of the $\alpha_0$ hyperplanes that contain no $a$-subspace is a hyperplane that contains $W_t$ but not $W_a$.
As there are $\theta_{n-t}-\theta_{n-t-a}$ such hyperplanes, it follows that
\[
\alpha_0=\theta_{n-t}-\theta_{n-t-a}.
\]
Since we considered any $a$-subspace of $ST$, the argument shows that the $\theta_{n-t-a}$ hyperplanes that contain some $W_t$ must indeed contain all the $a$-spaces of $ST$. Thus all the subspaces of the supertail must be contained in the intersection $T$ of these $\theta_{n-t-a}$ hyperplanes. Moreover, since $\theta_{n-t-a}$ hyperplanes intersect in a subspace of dimension at most $t+a$, it follows that $T$ has dimension $t+a$ and is thus partitioned by the supertail.
\end{proof}

\bs We are now ready to prove our main theorem.
\begin{proof}[Proof of Theorem~\ref{thm:main}]\
Let $\Pa$ be a partition of $V(n,q)$ with a $d_s$-supertail $ST$ of size 
$|ST|=\sigma_q(d_s,d_{s-1})=q^{d_{s-1}}+1$. To simplify the notation, we set $d=d_s$ and $t=d_{s-1}$.
Let $k$ and $r_d$ be integers such that $k\geq1$, 
$n=kd+r_d$,  and $1\leq r_d< d$.
Recall that if $d\geq 2t$, then Corollary~\ref{cor:HLNS2} holds.
So we may assume that $r_t=d-t$ satisfies $0<r_t<t$. 

Since $\Pa$ contains subspaces of dimensions $d$ and $t$, it follows that $n\geq d+t$. 
 We now show that $n\geq 2d$. 
By way of contradiction, assume that $n< 2d$.
Then the dimension condition (see~\eqref{eq:dim} in Section~\ref{sec:1})
implies that $\Pa$ contains at most one $d$-subspace. Thus, $\Pa$ contains exactly one 
$d$-subspace, $Y$, and  its $d$-supertail is $ST=\Pa\setminus \{Y\}$. 
Since $n\geq d+t$, $|ST|=q^t+1$ and the maximum dimension of any subspace in $ST$ is (by definition) $t$, 
it follows that
\begin{align}\label{eq:main-thm-1}
q^t+1=|ST|\geq \frac{\left|\sum\limits_{X\in ST}|X^*|\right|}{|V(t,q)^*|}
=\frac{|V(n,q)^*|-|Y^*|}{|V(t,q)^*|}=\frac{(q^n-1)-(q^{d}-1)}{q^t-1}\geq q^{d},
\end{align}
which is a contradiction since $d>t$ and $q\geq2$. Thus, a $d$-supertail of a partition $\Pa$ of $V(n,q)$ cannot
be of minimum size $\sigma_q(d,t)=q^t+1$ if $n<2d$. 
So we may assume that $n\geq 2d$, i.e.,  $k\geq 2$
 (which we will use in the proof of part $(iii)$ below).
We now prove the theorem for each of the three conditions $(i)$, $(ii)$, and $(iii)$ stated in the theorem.

\bs\n $(i)$
Suppose the supertail $ST$ contains subspaces of  at most two different dimensions $d_1=a$ and $d_{s-1}=t$ such that $t>a$. Since $n_i$ denotes the number of $i$-subspaces, we have
\begin{equation}\label{eq:B.1}
n_t+n_a=|ST|=q^t+1, 
\end{equation}
 where $n_t>0$ and $n_a\geq 0$.
Moreover, since $d=t+r_t$, Lemma~\ref{lem:A2} yields 
\begin{equation}\label{eq:B.2}
n_t\ta_t +n_a\ta_a =\frac{c_0q^d-1}{q-1},
\end{equation}
where $c_0$ is a positive integer.
Since $\ta_i=(q^i-1)/(q-1)$, it follows from~\eqref{eq:B.1} and~\eqref{eq:B.2} that 
\[n_t(q^{t-a}-1)=
q^t(c_0q^{r_t-a}-1)+q^{t-a}-1\Rightarrow n_t=\frac{q^t(c_0q^{r_t-a}-1)}{q^{t-a}-1}+1.\]

Since $\gcd(q^t,q^{t-a}-1)=1$, the above equation implies that $q^{t-a}-1$ divides $c_0q^{r_t-a}-1$.
Hence $n_t=q^t\cdot x+1$, where 
$x=\ds\frac{c_0q^{r_t-a}-1}{q^{t-a}-1}$
 is either $0$ or $1$ since 
$n_t\leq q^t+1$.
If $x=0$, then $n_t=1$ and $n_a=q^t$.
In this case, it follows from Lemma~\ref{lem:sub} that the union of the subspaces in $ST$ is a subspace 
of dimension $t+a=d_{s-1}+d_1$.
If $x=1$, then $n_t=q^t+1$ and $n_a=0$. 
In this case, $ST$ contains only subspaces of dimension $t=d_{s-1}$ and Theorem~\ref{thm:He-T} implies 
that $ST$ is a $d_{s-1}$-spread.

\bs\n $(ii)$ If $t-r_t=1$, then it follows from Lemma~\ref{lem:A} that the smallest 
dimension in $ST$ is $d_1\geq r_t=t-1$. Thus $ST$ contains  subspaces of at most two different 
dimensions, namely $t$ and $t-1$.
Now the main theorem follows from Theorem~\ref{thm:He-T} and part $(i)$ above.

\bs\n $(iii)$ Recall that $d=d_s$, $t=d_{s-1}$, $r_t=d-t$, and $n\geq 2d$. 
Moreover, $k$ and $r_d$ are integers such that $k\geq2$, $n=kd+r_d$, and $1\leq r_d< d$. 
Let $\ell=q^{r_d} \sum_{i=0}^{k-2}q^{id}$, and let $\Pa$ be a partition of $V(n,q)$ with a ${d}$-supertail 
$ST$ of minimum size $q^t+1$. 
Then $\Pa\setminus ST$ must be a partial $d$-spread. Note that
\begin{equation}\label{psp:eq1}
\ell =q^{r_d} \sum\limits_{i=0}^{k-2}q^{id}\Rightarrow \ell q^{d}=\frac{q^{d}(q^{n-{d}}-q^{r_d})}{q^{d}-1}.
\end{equation}
Let $\mu_q(n,d)$ denote the maximum number of $d$-subspaces in any partial 
$d$-spread of $V(n,q)$.
Then the upper bound given by Drake-Freeman~\cite[Corollary~$8$]{DrFr} implies  that 
\begin{equation}\label{psp:eq2}
\mu_q(n,d)<\ell q^{{d}}+\frac{q^{r_d}+q^{r_d-1}}{2}+1,
\end{equation}
Since $n=kd+r_d$ 
and the maximum dimension in $\Pa$ is ${d}$, it follows from Theorem~\ref{thm:HLNS1} that 
\begin{equation}\label{psp:eq2.2}
|\Pa|\geq \ell q^{d}+q^{\lceil \frac{{d}+r_d}{2}\rceil}+1.
\end{equation}
By definition of $\Pa$ and $ST$, it follows that
\begin{equation}\label{psp:eq2.5}
|\Pa|=|\Pa\setminus ST|+|ST|=n_{d}+q^t+1.
\end{equation}
Hence,~\eqref{psp:eq2},~\eqref{psp:eq2.2}, and~\eqref{psp:eq2.5} yield
\begin{equation}\label{psp:eq3}
q^{\lceil \frac{{d}+r_d}{2}\rceil}-q^t\leq n_{d}-\ell q^{d}<\frac{q^{r_d}+q^{r_d-1}}{2}+1.
\end{equation}
Since $0\leq r_d<{d}$, $d=t+r_t$, and $1\leq r_t<t$, it follows that 
\begin{equation}\label{psp:eq4}
-q^{d}<q^{\lceil \frac{{d}+r_d}{2}\rceil}-q^t \mbox{ and } \frac{q^{r_d}+q^{r_d-1}}{2}<q^{r_d}<q^{d}.
\end{equation}
Thus,~\eqref{psp:eq3} and~\eqref{psp:eq4} yield
\begin{equation}\label{psp:eq5}
-q^{d}< n_{d}-\ell q^{d}<q^{d}.
\end{equation}
Next, it follows from Lemma~\ref{lem:A2} that there exists some integer $c_0$ such that the number of vectors in $\bigcup_{X\in ST}X$ is equal to 
\begin{equation}\label{psp:eq6}
\sum_{i=a}^t n_i(q^i-1)=c_0q^{d}-1,
\end{equation}
where $a$ is the smallest dimension of a subspace in $ST$.
Since $\Pa\setminus ST$ is a partial ${d}$-spread of $V(n,q)$, it follows from~\eqref{psp:eq6} and from
counting in two ways the number of nonzero vectors in $V(n,q)$ that 
\begin{equation}\label{psp:eq7}
n_{d}(q^{d}-1)+c_0q^{d}-1=q^{n}-1 \Rightarrow n_{d}(q^{d}-1)=q^{d}(q^{n-{d}}-c_0).
\end{equation}
Hence,~\eqref{psp:eq7} implies that $q^{d}$ divides $n_{d}$. Thus $q^{d}$ divides $n_{d}-\ell q^{d}$. 
Now the second inequality  in~\eqref{psp:eq5} implies that $n_{d}-\ell q^{d}=0$, i.e.
\begin{equation}\label{psp:eq7.5}
|\Pa\setminus ST|=n_{d}=\ell q^{d}.
\end{equation}
Since $n_{d}=\ell q^{d}$ and ${d}=t+r_t$, it follows from the first inequality in~\eqref{psp:eq3} that 
\[ q^{\lceil \frac{{d}+r_d}{2}\rceil}-q^t\leq n_{d}-\ell q^{d}=0\Rightarrow r_d\leq t-r_t
\]
Since $d=t+r_t$, $|ST|=q^t+1$, $n_{d}=\ell q^{d},$ and $r_d\leq t-r_t$, it follows from Lemma~\ref{lem:HLNS3-2}
(see the Appendix) that $W=\bigcup_{X\in ST}X$ is 
a subspace of dimension ${d}+r_d=t+r_t+r_d\leq 2t$. 
If $r_d=t-r_t$, then $\dim W=2t$. Then $ST$ is a subspace partition of $W$ 
into $t$-subspaces only, since otherwise,
counting the nonzero vectors in $W$ in two different ways, i.e., 
\[q^{2t}-1=|W^*|=\sum_{X\in ST}|X^*|<|ST|\cdot |V(t,q)^*|=q^{2t}-1,\]
yields a contradiction. 
 We remark that if $r_d= t-r_t$, then $ \frac{{d}+r_d}{2}=t$; thus the subspace partition $\Pa$ must 
 be of type $[d^{\ell q^d}, t^{q^{t+1}}]$ and of minimum size (i.e., $|\Pa|=\sigma_q(n,d)$).
Such partitions  are known to exist and are discussed in~\cite{HLNS3}
(in particular, see Theorem~$2$).

Finally, if $r_d<t-r_t$, then $\dim W=t+r_t+r_d<2t$ and it follows from the dimension condition (see~\eqref{eq:dim} in Section~\ref{sec:1}) that $ST$ is a subspace partition of $W$ with exactly one $t$-subspace and with other subspaces of dimension at most $r_t+r_d$. Since $\dim W=t+r_t+r_d<2t$ and $|ST|=q^t+1$, counting in two ways 
the nonzero vectors in $W$ yields 
\begin{align}\label{eq:iii-1}
q^{t+r_t+r_d}-1=|W^*|
&=\sum_{X\in ST}|X^*|\cr
&\leq |V(t,q)^*|+(|ST|-1)\cdot|V(r_t+r_d,q)^*|=q^{t+r_t+r_d}-1.
\end{align}
If some subspace in $ST$ has dimension less than $r_t+r_d$, then inequality in~\eqref{eq:iii-1} 
becomes strict and we obtain a contradiction. Thus, $ST$ contains one $t$-subspace and $q^t$ subspaces of dimension $r_t+r_d$. 
In this case, we also remark that if $r_d= t-r_t-1$, then the subspace partition $\Pa$ is 
of type $[d^{\ell q^d}, t^1, (t-1)^{q^t}]$ and of minimum size
(e.g., see Theorem~$4$ in~\cite{HLNS3} for the existence of such partitions). 
However, if $r_d< t-r_t-1$, then the resulting subspace partitions are not 
necessarily of minimum size. For instance, if $n=34$, 
$k=3$, $d=11$, $r_d=1$, $t=7$, and $r_t=4$, we can 
apply (several times) Lemma~\ref{lem:Be} to construct a subspace partition of $V(34,q)$ of type
$[11^{q^{23}+q^{12}},7^1,5^{q^7}]$ and size $q^{23}+q^{12}+q^7+1$, which is 
larger than $\sigma_q(34,11)=q^{23}+q^{12}+q^6+1$.  
\end{proof}
We are now ready to prove Corollary~\ref{cor:main-cor}.
\begin{proof}
Let $\Pa$ be a subspace partition of $V=V(n,q)$ of type $[d_1^{n_1},\ldots,d_m^{n_m}]$.
Let $ST$ be the $d_s$-supertail of $\Pa$, let $\widehat{ST}$ be its $d_{s+1}$-supertail, and
assume that $ST$ has size $\sigma_q(d_s,d_{s-1})$. We shall prove the following statements:
\begin{enumerate}
\item[(i)] If $2\leq s\leq3$, $d_s<2d_{s-1}$, and $d_{s+1}< 2d_s$, then
\[|\widehat{ST}|\geq \sigma_q(d_{s+1},d_{s})+\sigma_q(d_s,d_{s-1}).\]
\item[(ii)]If $d_s\geq 2d_{s-1}$, or if $s=3$ and $d_3=d_2+d_1$, then
\[|\widehat{ST}|\geq \sigma_q(d_{s+1},d_{s})+\sigma_q(d_s,d_{s-1})-1.\]
\end{enumerate}

By using the hypotheses of 
 Corollary~\ref{cor:HLNS2} and Theorem~\ref{thm:main}, we can infer that $W=\bigcup_{X\in ST}X$ is a subspace of $V$ 
and $\dim W\geq d_s$. Thus, it follows from the definitions of $\Pa$ and $ST$ that 
\begin{equation}\label{def:Pp}
\Pa'=\left\{X\in\Pa:\, X\not\in ST \right\}\cup\{W\}
\end{equation}
is a subspace partition of $V$. Since $\Pa$ contains a $d_s$-subspace and $\dim X<d_s$ for any $X\in ST$, 
it follows from the definition in~\eqref{def:Pp} that $\Pa'$ also contains a $d_s$-subspace $Y$. 
Let $\widehat{ST'}$ be the $d_{s+1}$-supertail of $\Pa'$ and let $h$ be the the largest dimension of a subspace
 in $\widehat{ST'}$. Since $Y\in\widehat{ST'}$, it follows that 
$d_s=\dim Y\leq h<d_{s+1}$.  

 If $W\not\in\widehat{ST'}$, then $h=\dim Y=d_s$, and $\widehat{ST}$ is the disjoint union of $\widehat{ST'}$ 
and $ST$. Thus, 
\begin{align}\label{eq:STh1}
|\widehat{ST}|=|\widehat{ST'}|+|ST|\geq \sigma_q(d_{s+1},d_s)+\sigma_q(d_s,d_{s-1}).
\end{align}
Observe that $W\not\in\widehat{ST'}$ if and only if $\dim W\geq d_{s+1}$. 
In this case, $d_{s+1}\leq \dim W<2d_s$. 

If $W\in\widehat{ST'}$, then $h=\dim W<d_{s+1}$, and $\widehat{ST}$ is the disjoint union of 
$\widehat{ST'}\setminus\{W\}$ and $ST$. Thus, 
\begin{align}\label{eq:STh2}
|\widehat{ST}|\geq (|\widehat{ST'}|-1)+|ST|
&\geq \sigma_q(d_{s+1},\dim W)+\sigma_q(d_s,d_{s-1})-1.
\end{align}
If $d_s<2d_{s-1}$ and $d_{s+1}< 2d_s$, then $\dim W< d_{s+1}< 2\dim W$, and
it follows from Theorem~\ref{thm:HLNS1} that 
\begin{align}\label{eq:STh3}
\sigma_q(d_{s+1},\dim W)=q^{\dim W}+1>q^{d_s}+1=\sigma_q(d_{s+1},d_s).
\end{align}
Thus, from~\eqref{eq:STh2} and the strict inequality of ~\eqref{eq:STh3}, we obtain that  
\begin{align}\label{eq:STh4}
|\widehat{ST}|\geq \sigma_q(d_{s+1},d_s)+\sigma_q(d_s,d_{s-1}).
\end{align}
Now part $(i)$ of the corollary follows from~\eqref{eq:STh1}, \eqref{eq:STh3}, 
and~\eqref{eq:STh4}.

Finally, $\dim W=d_{s}$ if $d_s\geq 2d_{s-1}$ (by Corollary~\ref{cor:HLNS2}), or if $d_s=d_1+d_{s-1}$
(by Theorem~\ref{thm:main}). Thus part $(ii)$ of the corollary follows from~\eqref{eq:STh2}. 
\end{proof}

\section{Appendix}\label{sec:4}

\bs The following two lemmas (Lemma~\ref{lem:HLNS3-1} and Lemma~\ref{lem:HLNS3-2}) are direct  
adaptations of~\cite[Lemma~$7$ and Proposition~$1$]{HLNS3}. 
For the sake of completeness, we repeat their proofs here. 
In the following, let $\mathcal D$ denote the family of $d$-subspaces in a partition $\Pa$ of 
$V=V(n,q)$ with minimum size $d$-supertail $ST$. Let $\alpha_i$ denote the number of hyperplanes in $V$ that 
contain exactly $i$ members of 
$\mathcal D$.
\begin{lemma}\label{lem:HLNS3-1}
Let $n$, $k$, ${d}$, and $r_d$ be integers such that $k\geq2$, $n=kd+r_d$, ${d}=t+r_t$, $1\leq r_d<d$, and $1\leq r_t<t$.
Let $\ell=q^{r_d} \sum_{i=0}^{k-2}q^{id}$, and let
$\Pa$ be a partition of $V=V(n,q)$ whose largest subspace dimension is $d$ and such that $n_{d}=\ell q^{d}$.
Assume, furthermore that $\Pa$ has a $d$-supertail $ST$ of minimum size $|ST|=q^t+1$ and 
with largest subspace dimension $t$.
Then, the following conclusions hold.
\begin{enumerate}
\item[(a)]  If $\alpha_i\neq0$, then $\delta=\ell-q^{r_d} \leq i\leq \ell$. 
\item[(b)] The extremal case $\alpha_\delta\neq0$ occurs if there exists a hyperplane $H$ of $V$ such that 
all members of $ST$ are subspaces of $H$.
\item[(c)] The extremal case $\alpha_\ell \neq 0$ occurs if there exists a hyperplane $H$  of $V$ such that the 
number of non-zero vectors in $\bigcup_{X\in ST}(X\cap H)$ equals $q^{d+r_d-1}-1$.
\end{enumerate}  
\end{lemma}
\begin{proof}
For any subspace $U$ of $V$ and any hyperplane $H$ of $V$, let $B_H(U)$ denote the set of all points 
(i.e., $1$-subspaces) of $U$ that are not in $H$. 
Then elementary linear algebra arguments yield
\begin{equation}\label{eq:ny2}
|B_H(U)|=
\left\lbrace\begin{array}{lll}
0            &\hbox{if $U\subseteq H$},\\
q^{\dim U-1}&\hbox{otherwise}.
\end{array}\right.
\end{equation}
If $\alpha_i\neq 0$ then there is at least one hyperplane $H$ containing $i$ members of $\Da$. 
From~\eqref{eq:ny2} we then get that 
\[(n_{d}-i)q^{d-1}\leq |B_H(V)|=q^{n-1}\]
and thus, since $n=kd+r_d$ and $\ell=q^{r_d} \sum_{i=0}^{k-2}q^{id}$, we obtain
\begin{equation}\label{eq:ny3}
i\geq n_{d}-q^{(k-1){d}+r_d}=\ell q^{d}-q^{(k-1){d}+r_d}=\ell-q^{r_d}. 
\end{equation}
We now show  by contradiction that $i\leq \ell$ for all $\alpha_i\not=0$. Assume that $i\geq \ell+1$ for some $H$.
Since $\Da$ denotes the set of members of $\Pa$ that have dimension ${d}$, we have 
	$|\Pa\setminus\Da|=q^t+1$.
Then it follows from Lemma~\ref{lem:HeLe0} and the fact that ${d}>t$ that 
\begin{equation}\label{eq:ny4}
|\Pa|\geq i\cdot q^{d}+1\geq (\ell+1)q^{d}+1>\ell q^{d}+q^t+1=|\Da|+|\Pa\setminus\Da|=|\Pa|,
\end{equation}
which is a contradiction.  Now conclusion $(a)$ of the theorem follows from~\eqref{eq:ny3} and~\eqref{eq:ny4}. 

Next, we can infer from the  analysis leading to ~\eqref{eq:ny3} 
that the case $\alpha_\delta\neq0$ with $\delta=\ell-q^{r_d}$ occurs if all members of  $ST$ are contained in some  hyperplane $H$. This proves conclusion $(b)$.
Finally, if  $\alpha_\ell \neq 0$ occurs, then by definition of $\alpha_i$, there exists a hyperplane $H$ of $V$
that contains exactly $\ell$ subspaces of $\Da$.
Let $\Da'\subseteq \Da$ be the set containing those $\ell$
subspaces. Since $\Pa$ is a subspace partition of $V$, counting the nonzero vectors of $H$ in two ways yields
\begin{equation}\label{eq:ny5}
|H^*|=\left|\bigcup_{X\in\Da'}(X\cap H)^*\right|+ \left|\bigcup_{X\in\Da\setminus\Da'}(X\cap H)^*\right|+
\left|\bigcup_{X\in ST}(X\cap H)^*\right|,
\end{equation}
where $(X\cap H)^*$ is the set of nonzero vectors in $X\cap H$. 
Since $\Da$ contains $\ell q^{d}$ subspaces of dimension $d$, $\dim H=kd+r_d-1$, and $\dim(X\cap H)=d-1$ for all 
$X\in(\Da\setminus\Da')$, it follows from~\eqref{eq:ny5}
that 
\[\left|\bigcup_{X\in ST}(X\cap H)^*\right|=(q^{kd+r_d-1}-1)-\ell(q^{d}-1)-(\ell q^d-\ell)(q^{d-1}-1)=q^{d+r_d-1}-1,\]
where we used the fact that $\ell=q^{r_d} \sum_{i=0}^{k-2}q^{id}=q^{r_d}\frac{q^{(k-1){d}}-1}{q^d-1}$. 
This proves conclusion $(c)$.
\end{proof}
\begin{lemma}\label{lem:HLNS3-2}
Let $n$, $k$, ${d}$, and $r_d$ be integers such that $k\geq2$, $n=kd+r_d$, ${d}=t+r_t$, $1\leq r_d<d$, and $1\leq r_t<t$.
Let $\ell=q^{r_d} \sum_{i=0}^{k-2}q^{id}$, and let
$\Pa$ be a partition of $V(n,q)$ whose largest subspace dimension is $d$ and such that $n_{d}=\ell q^{d}$.
Assume, furthermore that $\Pa$ has a $d$-supertail $ST$ of minimum size $|ST|=q^t+1$ and 
with largest subspace dimension $t$.
Then the union of the subspaces in $ST$ is itself a 
$d+r_d$-subspace.
\end{lemma}
\begin{proof}  
We again let $\Da$ denote the family of $d$-subspaces in $\Pa$, and 
we let $\alpha_i$ denote the number of hyperplanes in $V=V(n,q)$ that contain exactly 
$i$ members of $\Da$.
It is trivial that $\alpha_i\geq0$ for all $i$; a fact that we will use later for all $i$. 
From Lemma \ref{lem:HLNS3-1}, we know that 
\[\alpha_i\neq0\qquad\Longrightarrow\qquad \delta=\ell-q^{r_d}\leq i\leq \ell.\]

We first define the integers $x$, $y$ and $z$ by
\[
x=\sum_{i=\delta}^\ell i\alpha_i,\qquad y=\sum_{i=\delta}^{\ell}{i\choose2}\alpha_i,
\mbox { and } z=\sum_{i=\delta}^{\ell} \alpha_i.
\]
Each member of $\Da$ is a $d$-subspace 
and is thus contained in exactly $(q^{(k-1){d}+r_d}-1)/(q-1)$ 
hyperplanes. By double counting incidences $(H,U)$, for $H\in\Ha$ with $U\in\Da$ and 
$U\subseteq H$, we obtain 
\begin{equation}\label{eq:x}
x=\sum_{i=\delta}^\ell i\alpha_i=n_{d}\cdot\ta_{(k-1)d+r_d}.  
\end{equation} 
Any two members of $\Da$ are contained in $(q^{(k-2){d}+r_d}-1)/(q-1)$ hyperplanes. 
Thus, by double counting incidences, we get
\begin{equation}\label{eq:y}
y=\sum_{i=\delta}^{\ell}{i\choose2}\alpha_i={n_{d}\choose2}\ta_{(k-2)d+r_d}. 
\end{equation}
Furthermore, by counting the number of hyperplanes in $V$, we obtain
\begin{equation}\label{eq:z}
z=\sum_{i=\delta}^{\ell} \alpha_i=\ta_{kd+r_d}. 
\end{equation}
Observe that~\eqref{eq:x},~\eqref{eq:y}, and~\eqref{eq:z} imply that the constants 
$x$, $y$ and $z$ are independent of the particular choice of subspace of $\Pa$ with $n_{d}=\ell q^{d}$ 
and minimum size ${d}$-supertail. Moreover,
\begin{equation}\label{eq:ABC-1}
\sum_{i=\delta}^{\ell}\alpha_i(i-\delta)(i-\ell)=2y+x-(\delta+\ell)x+\delta\ell z.
\end{equation}
Also note the following facts that we shall use later:
\begin{equation}\label{eq:ABC-2}
(i-\delta)(i-\ell)
\begin{cases}
=0\;&\mbox{ if } i=\delta,\\
<0\;&\mbox{ if } \delta<i<\ell,\\
=0\;&\mbox{ if } i=\ell.
\end{cases}
\end{equation}

Since $n=kd+r_d$, ${d}=t+r_t$, and $1\leq r_d\leq t-r_t$, we can use Lemma~\ref{lem:Be} to construct a partition 
$\Pa_0$ of $V(n,q)$ with $\ell q^{d}$ subspaces of dimension ${d}$, one $t$-subspace, and $q^t$ subspaces of dimension $r_d+r_t$. (Note that if $r_d=t-r_t$, $\Pa_0$ has type 
$[{d}^{\ell q^{d}},t^{q^t+1}]$.)

In order to show that the right side of \eqref{eq:ABC-1} is equal to zero, we consider the  
partition $\Pa_0$. From the construction of the partition ${\mathcal P}_0$, it follows that the points in the 
$d$-supertail constitute a ${d}+r_d$-subspace $W$.
Any hyperplane $H\in\mathcal{H}$ either contains $W$ or intersects $W$ in $(q^{\dim W-1}-1)$ non-zero vectors. 
These are the two extremal cases discussed in parts~$(b)$ and $(c)$ of  Lemma~\ref{lem:HLNS3-1}. 
So for the partition $\Pa_0$, we have $\alpha_i=0$ for $\delta<i<\ell$. Then, it follows from~\eqref{eq:ABC-2}
that the left side of~\eqref{eq:ABC-1} is equal to zero. Thus, we obtain from~\eqref{eq:ABC-1} that for any partition $\Pa$,
\[ \sum_{i=\delta+1}^{\ell-1}\alpha_i(i-\delta)(i-\ell)=0.\]
As $\alpha_i\geq0$, we may thus conclude from the equation above and~\eqref{eq:ABC-2} that 
\[\delta< i< \ell\qquad\Longrightarrow\qquad \alpha_i=0.\]
Hence, we can now use~\eqref{eq:x} and~\eqref{eq:z} (or refer to the partition $\Pa_0$,
which must have the same solution $\alpha_\delta$ and $\alpha_\ell$ to these two equations) 
to calculate $\alpha_\delta$ (and $\alpha_\ell$). We then get that 
\begin{equation}\label{eq:xc}
\alpha_\delta=\ta_{(k-1)d}.
\end{equation}

Let $\gamma=q^{(k-1){d}+r_d}$. Since 
\[\ell=q^{r_d} \sum_{i=0}^{k-2}q^{id}=q^{r_d}\frac{q^{(k-1){d}}-1}{q^d-1},\]
it follows that
\begin{equation}\label{eq:H0}
|\Da|-\gamma=\ell q^d-q^{(k-1){d}+r_d}=\ell-q^{r_d}=\delta.
\end{equation}
Let $\mathcal{H}_0$ denote the set of all hyperplanes of $V$ that intersect 
$\gamma$ members of $\Da$. 
Since a hyperplane of $V$ either contains a given subspace or intersects it, 
it follows from~\eqref{eq:H0} that $\mathcal{H}_0$ can be also defined as the set of all hyperplanes of $V$ that 
contain $\delta$ members of $\Da$. Thus it follows from the definition of $\alpha_i$ that $\alpha_\delta=|\Ha_0|$.
Let 
\[W=\bigcap_{H\in\mathcal{H}_0} H\;\mbox{ and }\; R=\bigcup_{X\in ST}
X.\]
Since $\Pa$ is a subspace partition of $V$ and $\Da$ contains $\ell q^d$ subspaces of dimension 
$d$, the number of non zero vectors in $R$ is
\begin{equation}\label{eq:S}
|R^*|=|V^*|-\sum_{U\in\Da}U^*=q^n-1-\ell q^{d}(q^{d}-1)=q^{{d}+r_d}-1. 
\end{equation}
Moreover, it follows from Lemma~\ref{lem:HLNS3-1}~$(b)$ that
\begin{equation}\label{eq:SW}
R\subseteq\bigcap_{H\in\mathcal{H}_0} H =W,
\end{equation} 
and it follows from~\eqref{eq:xc} that
\begin{equation}\label{eq:W}
\dim W\leq n-(k-1){d}={d}+r_d.
\end{equation} 
Thus, it follows from~\eqref{eq:S}--\eqref{eq:W} that $R=W$ is a $d+r_d$-subspace.
\end{proof}

\bs
\paragraph{}
\n{\bf Acknowledgement:}
We thank the referees for their valuable comments which helped 
improve the presentation and simplify some proofs (in particular, Lemma~\ref{lem:sub}).

\end{document}